\documentclass[10pt]{amsart}
\usepackage{amsmath, amssymb, amsfonts}
 \usepackage{mathrsfs}
  \usepackage{bbm}
 \usepackage{mathabx}

\newcommand{\no}[1]{#1}
\renewcommand{\no}[1]{}
\no{\usepackage{times}\usepackage[subscriptcorrection, slantedGreek, nofontinfo]{mtpro}
\renewcommand{\Delta}{\upDelta}}
\usepackage{color}

\date{\today}
\def\eps{\varepsilon}
\def\RR{{\rm I~\hspace{-1.15ex}R} }
\def\CC{\rm \hbox{C\kern-.56em\raise.4ex
         \hbox{$\scriptscriptstyle |$}\kern+0.5 em }}
\newcommand{\be}{\begin{equation}}
\newcommand{\ee}{\end{equation}}
\newcommand{\ba}{\begin{equation*}}
\newcommand{\ea}{\end{equation*}}
\newcommand{\beq}{\begin{eqnarray}}
\newcommand{\eeq}{\end{eqnarray}}
\newcommand{\beqs}{\begin{eqnarray*}}
\newcommand{\eeqs}{\end{eqnarray*}}

\newcommand{\rfb}[1]{\mbox{\rm
   (\ref{#1})}\ifx\undefined\stillediting\else:\fbox{$#1$}\fi}
\newtheorem{lemma}{Lemma}[section]
\newtheorem{remark}{Remark}[section]
\newtheorem{definition}{Definition}[section]

\newtheorem{proposition}{Proposition}[section]
\newtheorem{theorem}{Theorem}[section]

\newcommand{\cc}{\mathfrak C}

\title[Weak Observability]{On Weak Observability  For Evolution Systems with Skew-Adjoint Generators}

\author{Ka\"{\i}s Ammari\dag}
\address{\dag Universit\'e de Monastir, Facult\'e des Sciences  de Monastir, D\'epartement de Math\'ematiques, UR Analyse et Contr\^ole des EDP, UR 13ES64, 5019, Monastir, Tunisia}
\email{kais.ammari@fsm.rnu.tn}

\author{Faouzi Triki\ddag}
\address{\ddag Laboratoire Jean Kuntzmann,  UMR CNRS 5224, 
Universit\'e  Grenoble-Alpes, 700 Avenue Centrale,
38401 Saint-Martin-d'H\`eres, France}
\email{faouzi.triki@univ-grenoble-alpes.fr}

\date{\today}

\begin{document}

\begin{abstract}
In the paper we consider the linear  inverse problem  that consists in recovering the initial  state in a first order 
evolution equation generated by a skew-adjoint operator. We studied  the well-posedness  of the inversion
in terms of the observation operator and the spectra of the skew-adjoint generator.
The  stability estimate of the inversion  can also 
be seen as  a weak observability inequality. The proof of 
the main results is based on a new resolvent inequality and Fourier transform techniques which are of interest  themselves. 
\end{abstract}
 

\maketitle

\tableofcontents

\vfill\break
\section{Introduction} \label{sec1}
\setcounter{equation}{0}

Let $X$ be a complex Hilbert space with norm and inner product 
denoted respectively by $\|\cdot\|_{X}$ and  $\langle \cdot,\cdot \rangle_X$.
Let $A: X\rightarrow X$ be a linear unbounded self-adjoint, strictly 
positive operator with a compact resolvent. Denote by 
${ D}(A^{\frac{1}{2}})$ the domain of 
$A^{\frac{1}{2}}$, and  introduce  for $\beta \in \mathbb R$ the scale of Hilbert
spaces $X_{\beta}$, as follows: for every
$\beta \geq 0$, $X_{\beta}= {D}(A^{\frac{\beta}{2}})$, with the norm
$\|z \|_\beta=\|A^{\frac{\beta}{2}} z\|_X$ (note that 
$0 \notin \sigma(A) $ where $\sigma(A)$
is the spectrum of $A$). The space $X_{-\beta}$ is
defined by duality with respect to the pivot space $X$ as
follows: $X_{-\beta} =X_{\beta}^*$ for $\beta>0$.\\
The operator $A$ can be extended (or restricted) to each $X_\beta$,
such that it becomes a bounded operator
\be
\label{A0extb}
A : X_\beta \rightarrow X_{\beta-2}\;\;\; \forall \beta \in \mathbb R.
\ee

\medskip

The operator $iA$ generates a strongly continuous group of isometries in $X$
 denoted $(e^{itA})_{t\in \mathbb{R}}$ \cite{TW09}.

Further, let $Y$ be a complex Hilbert space (which will be identified to its dual space) with 
norm and inner product respectively denoted by $||.||_{Y}$ and 
$\langle \cdot,\cdot \rangle_{Y}$, and let 
$C \in {\mathcal L}(X_{2}, Y)$, the space of linear bounded operators from $X_2$ into $Y$.\\

This paper is concerned with the following abstract 
 infinite-dimensional dual observation system with an output $y \in Y$  
described by  the equations 
\be
\label{maineq}
\left\{
\begin{array}{lll}
\dot z(t)- i Az(t) = 0, \, t > 0,  \\
z(0) = z_0\in X, \\
y(t) = C z(t),  \, t > 0.
\end{array}
\right.
\ee

In inverse problems framework the system above 
is called   the direct problem, i.e, to determine  the observation
 $y(t)= C z(t)$ of the state 
$z(t)$  for  given initial state $z_0$ and  unbounded  operator $A$. 
The inverse problem is to recover the initial state $z_0$ from the knowledge
of the observation $y(t)$ for $t\in [0, T]$ where $T>0$  is chosen to be large enough.

\medskip 

Inverse problems for evolution equations driven by  numerous applications,  have been a very 
active area in mathematical and numerical research over the last decades \cite{Is}. They are 
 intrinsically difficult to solve: this fact is due in part to their very mathematical structure and  to the
  effect that generally only partial data is available. Many different   linear inverse problems  in evolution
    equations  related to data assimilation,
medical imaging, and geoscience,  may fit in the general formulation of the system~\eqref{maineq} (see for instance  \cite{Ya, AC1, AC3, ACT, ACT1, BaK, RT} and references therein). 
 
\medskip
    
The system \eqref{maineq} 
has a  unique weak solution $z\in C(\RR,X)$
defined by:
\begin{eqnarray}
\label{eqmildsol}
z(t)=e^{it A}z_{0}.
\end{eqnarray}
If $z_0$ is not in $X_2$, in general $z(t)$ does not belong 
to $X_2$, and hence the last equation in \eqref{maineq}
might not be defined. We further make the following additional admissibility assumption on  the observation operator $C$:
 $\forall T>0$, \; $\exists C_{T}>0$,
\begin{eqnarray} \label{admm}
\label{eqq}
\forall z_{0}\in  X_2,
\quad \int_{0}^{T}\|C e^{itA}z_{0}\|_Y^{2}dt \leq C_{T}\|{z_{0}}\|_X^{2}.
\end{eqnarray} 
We immediately deduce from the admissibility assumption that the map
from $X_2$ to $L^{2}_{loc}(\RR_+;Y)$  that assigns $y$ for each $z_0$, has a continuous extension to $X$.  Therefore the last equation in
 \eqref{maineq} is now well defined for  all $z_0\in X$. 
 Without loss of generality we assume that $C_T$ is an increasing function of $T$
 (if the assumption is not satisfied we substitute  $C_T$ by
  $ \sup_{0\leq t\leq T}C_T$).   

\medskip

Since  $A$ is  a self-adjoint operator with a compact resolvent, it follows that 
the spectrum of $A$ is given by  
$\sigma(A) \,=\,  \{ \lambda_k, \;  k\in \mathbb N^*\}$  where $\lambda_k$ is a sequence of 
strictly increasing  real numbers.  We denote    $(\phi_k)_{k\in
\mathbb N^*}$  the  orthonormal sequence of eigenvectors of $A$ associated to the eigenvalues 
$(\lambda_k)_{k\in
\mathbb N^*}$. 

\medskip
Let $ z\in X_2\setminus\{0\}\subset X \longmapsto \lambda(z)\in \mathbb R_+$ be the $A$-frequency 
function defined by
\begin{eqnarray}  \label{frequency}
\lambda(z) &=& \langle Az, z\rangle_X \|z\|_X^{-2},\\
&=& \sum_{k=1}^{+\infty} \lambda_k \langle z, \phi_k \rangle_X^2 
\left( \sum_{k=1}^{+\infty} \langle z, \phi_k \rangle_X^2 \right)^{-1}.
\end{eqnarray} 
We observe that $z \longmapsto \lambda(z)$  is continuous on $X_2\setminus\{0\}$, and 
$\lambda(\phi_k) = \lambda_k, \, k \in \mathbb N^*$.

\medskip

Let  $\cc $ be the set of  functions 
$\psi: \mathbb R_+ \rightarrow \mathbb R_+^*$  continuous and
decreasing. Recall that if $\psi \in \cc$ is not bounded below by a strictly
positive constant it satisfies 
$\lim_{t\rightarrow +\infty}\psi(t) = 0$.

\medskip

\begin{definition}
\label{ww} 
The system (\ref{maineq})
is said to be weakly observable  in time $T>0$  if there
exists $\psi \in \cc $ such that 
  following observation inequality holds:
\begin{eqnarray} 
\label{wobs}
\forall z_{0}\in X_2,
\quad
\psi(\lambda(z_0))\|z_{0}\|_X^{2}\leq 
\int_{0}^{T} \|C e^{itA}z_{0}\|_{Y}^2 dt.
 \end{eqnarray}
If  $\psi(t)$ is lower bounded, the system is said  to be exactly observable.
\end{definition}
\begin{remark} If the system  (\ref{maineq})
is  weakly observable  in time $T>0$, it is   weakly observable  in
any time $T^\prime$ larger than $T$.
The function $\psi$ appearing in the observability inequality \eqref{wobs} may depends
on the time $T$.  
\end{remark}

Most of the existing works on  observability  inequalities  for systems of partial differential equations are 
based on a time domain techniques as  nonharmonic series \cite{AI, KL},  multipliers method \cite{Li,Li1}, and 
microlocal analysis techniques \cite{BLR92, LL}. Only few of them  have considered  frequency domain 
techniques  in the spirit of the  well known Fattorini-Hautus test for  finite dimensional systems 
\cite{Fa, Ha,  BZbb, RTTT,  ZY97}. \\

The wanted  frequency domain test for the observability of the system (\ref{maineq})  would be only formulated
in terms of the operators $A$, $C$. The time domain system (\ref{maineq})  would be  converted into a frequency domain one, and the test would involve essentially  the solution in the frequency domain and  the observability operator $C$. The frequency domain test seems to be more suitable for numerical validation and  for the calibration  of  physical models
 for many reasons:  the parameters of the system are in general measured in frequency domain; the computation of the
 solution is more robust and efficient in frequency domain.

\medskip

The  objective  here is to derive sufficient  and if possible necessary 
 conditions on
 \begin{itemize}  
 
 \item[(i)] the spectrum of $A$, and 
 \item[(ii)]   on  the  action
 of the operator $C$  on the associated   eigenfunctions of $A$, 
 
 \end{itemize}
  such that  the closed  system \eqref{maineq}
verifies, for some $T>0,$ sufficiently  large, the inequality \eqref{wobs}. 
The aim of this paper is to obtain  Fattorini-Hautus type tests on the pair
$(A, C)$ that guarantee  the  {\it weak observability property} \eqref{wobs}.

\medskip

\medskip

The rest of the paper is organized as follows: In section \ref{sec2} we present the main results of our paper related 
to the weak observability. Section \ref{sec3} 
 contains the proof of the main Theorem \ref{main} based on new resolvent inequality and Fourier transform
 techniques. In section \ref{sec4} we study the relation between  the  spectral coercivity  of the observability operator and  his action on vector spaces spanned by eigenfunctions associated to close eigenvalues. Finally, in section \ref{sec5} 
 we apply the results of  the main  Theorem \ref{main}  to   boundary observability of the  Schr\"odinger equation in a square.

\section{Main results} \label{sec2}
We present in this section the  main results of our paper.

\begin{definition} \label{dwc}
The operator $C$ is spectrally coercive if 
there exist   functions
 $\varepsilon, \psi \in \cc$   such that if  $z\in X_2 \setminus \left\{0\right\}$
satisfies
\be \label{wc}
 \frac{\|Az\|_X^2}{\|z\|_X^2} -\lambda^2(z) <\varepsilon(\lambda(z)),
\ee
then
\be \label{wcc}
\|Cz\|_Y^2 \geq  \psi(\lambda(z))  \|z\|_{X }^2.
\ee

\end{definition}
\begin{remark}
We remark that the following relation
\be \label{magic}
 0 \leq \|(A-\lambda(z)I)z\|_X^2\|z\|_X^{-2} = \frac{\|Az\|_X^2}{\|z\|_X^2} -\lambda^2(z) 
 \ee
holds  for all $z\in X_2 \setminus\{0\}$. In addition, the equality 
 $\frac{\|Az\|_X^2}{\|z\|_X^2} -\lambda^2(z) =  0$ is satisfied if and only 
 if $z=\phi_k$ for some $ k\in \mathbb N^*$. 
 \end{remark}

Now, we are ready to announce our main result.
\begin{theorem}  \label{main}
The system \eqref{maineq} is weakly observable iff 
$C$ is spectrally coercive, that is the  following two assertions
are equivalent. 
 \begin{itemize}
 \item[(1)] 
 There exist  $\varepsilon, \psi \in \cc $  such that if  $z \in X_2 \setminus\left\{0 \right\}$
satisfying
\begin{eqnarray*}
 0\leq \frac{\|Az\|_X^2}{\|z\|_X^2} -\lambda^2(z) <\varepsilon(\lambda(z)),
\end{eqnarray*}
then
\begin{eqnarray*}
\|Cz\|_Y^2 \geq \psi(\lambda(z)) \|z\|_{X }^2.
\end{eqnarray*}

 \item[(2)] The  following weak observation inequality holds:
\begin{eqnarray} 
\label{wobst}
\forall z_{0}\in X_2,
\quad
\theta_2 \psi\left(\theta_0\left(\frac{1}{T}+\lambda(z_0)\right)\right)   \|z_0\|_X^2 \leq 
   \int_{0 }^T\|{C z(t)}\|^{2}_Yd\tau
\end{eqnarray}
for all $T\geq T(\lambda(z_0))$, where $T(\lambda(z_0))$ is the unique solution to
the equation
\beq \label{timeT}
T\varepsilon\left(\theta_0\left(\frac{1}{T}+\lambda(z_0)\right)\right)= \theta_1,
\eeq
 and $\varepsilon, \psi \in \cc $ are the functions appearing in the spectral coercivity
 of $C$.  The 
 strictly positives constants $\theta_i, \, i=0, 1, 2,$ do not depend on the parameters of the 
observability system.  In addition, the function  $\lambda \mapsto T(\lambda)$ is
 increasing.
\end{itemize}
\end{theorem}

The above theorem can  be viewed as a extension of several results in the literature \cite{Ha,  BZbb, RTTT,  ZY97, Miller05}. 
\section{Proof of the main Theorem~\ref{main}} \label{sec3}
In order to prove  our main theorem, we need to derive a sequence 
of preliminary results. We start with the  main tool in the proof  of the theorem which is a  generalized
Hautus-type test.  
 
\begin{theorem}
\label{hautus}
The operator $C\in \mathcal{L}(X_2, Y),$ is spectrally coercive, if and only if there exist functions $\psi,\, \varepsilon
\in \mathfrak C $, such that  the following  resolvent inequality  holds
\beq \label{resolventestimate}
\|z\|_X^2 \leq \inf\left\{\frac{\|{C z}\|^{2}_Y}{\psi(\lambda(z))}, 
\frac{\|{(A-\lambda I) z}\|^{2}_X}{(\lambda-\lambda(z))^2+\varepsilon(\lambda(z))} \right\}, \;\forall 
\lambda\in \RR,\; \forall z\in X_2\setminus\{0\}.
\eeq 

\end{theorem}
\begin{proof} 
 Let $z\in X_2\setminus\{0\}$ be fixed. A forward computation gives the following key identity:
 \beq \label{maineqh}
 \|{(A-\lambda I) z}\|^{2}_X = (\lambda-\lambda(z))^2 \|z\|_X^2 +
  \|{(A-\lambda(z) I) z}\|^{2}_X.
 \eeq
 We remark that the minimum of $\|{(A-\lambda I) z}\|^{2}_X$ for a fixed $z$ with respect to
 $\lambda \in \RR$ is reached at  $\lambda = \lambda(z).$\\
 
  We first assume that $C\in \mathcal{L}(X_2, Y),$ is spectrally coercive and prove 
  that \eqref{resolventestimate}
 is satisfied.  Let now  $ \varepsilon, \, \psi \in \mathfrak C $ the functions appearing in
  the spectral coercivity of the operator $C$ in Definition~\ref{dwc}, and  consider the following 
  two possible cases:\\ 

(i) The inequality 
$\|A z\|^2_X- \lambda^2(z)\|z\|_X^2<\varepsilon(\lambda(z)) \|z\|_X^2$ is satisfied. Then by the  
spectral coercivity of $C$, we deduce 
\begin{eqnarray} \label{i1}
\|Cz\|_Y^2 \geq \psi(\lambda(z)) \|z\|_{X}^2.
\end{eqnarray}

(ii) The inequality 
$\|A z\|^2_X- \lambda^2(z)\|z\|_X^2 \geq \varepsilon(\lambda(z)) \|z\|_X^2$ holds. Then,  the identity
\eqref{maineqh} implies 
\beq \label{i2}
\|{(A-\lambda I) z}\|^{2}_X \geq  \left((\lambda-\lambda(z))^2 + \varepsilon(\lambda(z))\right) \|z\|_X^2.
\eeq

By combining both inequalities \eqref{i1} and \eqref{i2},  we obtain the resolvent inequality 
\eqref{resolventestimate}.\\

We now assume that  \eqref{resolventestimate} holds and, we shall show that 
$C\in \mathcal{L}(X_2, Y),$ satisfies the spectrally coercivity in Definition~\ref{dwc}.
 Let   $ \varepsilon, \, \psi \in \mathfrak C $ the functions appearing in \eqref{resolventestimate}, 
 and  assume that $z \in X_2\setminus\{0\}$ satisfies 
  \begin{eqnarray} \label{j1}
\|{(A-\lambda(z) I) z}\|^{2}_X=  \|Az\|_X^2 -\lambda^2(z)\|z\|_X^2 <\varepsilon(\lambda(z))\|z\|_X^2.
\end{eqnarray}
Then, we have two possibilities 

(i) The inequality 
\beqs
\frac{\|{C z}\|^{2}_Y}{\psi(\lambda(z))} \leq 
\frac{\|{(A-\lambda I) z}\|^{2}_X}{(\lambda-\lambda(z))^2+\varepsilon(\lambda(z))} 
\eeqs
holds for some $\lambda \in \RR$.  Consequently the following  spectral coercivity 
\begin{eqnarray*} 
\|Cz\|_Y^2 \geq \psi(\lambda(z)) \|z\|_{X}^2
\end{eqnarray*}
can be  trivially deduced from the resolvent identity \eqref{resolventestimate}. \\

(ii) The inequality 
\beqs
\frac{\|{C z}\|^{2}_Y}{\psi(\lambda(z))} >
\frac{\|{(A-\lambda I) z}\|^{2}_X}{(\lambda-\lambda(z))^2+\varepsilon(\lambda(z))} 
\eeqs
is valid for all $\lambda \in \RR$. We then deduce from the identity
\eqref{maineqh}  the following inequality
\beqs
\frac{\|{C z}\|^{2}_Y}{\psi(\lambda(z))} >
\frac{(\lambda-\lambda(z))^2 \|z\|_X^2 +
  \|{(A-\lambda(z) I) z}\|^{2}_X}{(\lambda-\lambda(z))^2+\varepsilon(\lambda(z))}, \quad \forall
  \lambda \in \RR.
\eeqs
 Taking $\lambda $ to infinity we get the wanted inequality, that is 
 \beqs
\frac{\|{C z}\|^{2}_Y}{\psi(\lambda(z))} \geq  \|z\|_{X}^2,
\eeqs
 which finishes the proof of the Theorem.

 \end{proof}
Next we use a method developed  in \cite{BZbb} 
to derive observability inequalities based on resolvent 
inequalities and Fourier transform techniques.
Our objective is to prove the equivalence between 
the resolvent inequality 
\eqref{resolventestimate} and the weak observability 
\eqref{wobst}. The proof of the Theorem is then achieved 
by considering the results obtained in Theorem~\ref{hautus}.\\

We further assume that the   resolvent inequality 
\eqref{resolventestimate} holds and shall prove the weak observability.\\

Let $\chi \in C_0^\infty(\RR)$ be a cut off function with a
compact support in $(-1, 1)$. For $T>0$, we further denote 
\beq \label{cutoff}
\chi_T(t) &=& \chi \left(\frac{t}{T}\right), \qquad  t\in \mathbb R.
\eeq

Let $z_{0}\in X_2\setminus\{0\}$. 
Set $z(t)=e^{itA}z_{0}$, $x=\chi_T z$ and $f=\dot x -iA x$.
Since $\dot z-iA z =0$,  we have $f= \dot \chi_T z$.
The Fourier transform of $f$ with respect to time is given by
$$\widehat{f}(\tau)=(i\tau-iA)\widehat{x}(\tau),$$
where $\widehat{x}(\tau)$ is the Fourier transform of $x(t)$.
Applying \eqref{resolventestimate} to $\widehat{x}(\tau)\in X_2\setminus\{0\}$ for 
$\lambda= \tau$, we obtain
\be \label{inqq}
\|\widehat x(\tau)\|_X^2 \leq \inf\left\{\frac{\|{C \widehat x(\tau)}\|^{2}_Y}{\psi(\lambda(\widehat x(\tau)))}, 
\frac{\|\widehat{f}(\tau)\|^{2}_X}{(\tau-\lambda(\widehat x(\tau)))^2+\varepsilon(\lambda(\widehat x(\tau)))} \right\}.
\ee

We remark that since $\widehat x(\tau) \not= 0,$ we have  $\lambda(\widehat x(\tau))\not= +\infty,$ 
and the inequality \eqref{inqq} is well justified. Next, we study how do the frequency
 $\lambda(\widehat x(\tau))$ behave 
as a function of $\tau$. We expect that $\lambda(\widehat x(\tau)) $ that 
is close to $\lambda(z_0)$, the frequency of the initial state 
$z_0$, and reach  increases when  $|\tau|$  tends to infinity.\\

To simplify the analysis we will make some assumptions 
 on the cut-off function $\chi(s)$.  
We further assume that  $\chi \in C_0(\RR)$  satisfies
 the following inequalities: 
\beq \label{bbchi}
\chi \in H^1_0(-1,1),\;  \frac{\kappa_1}{1+\tau^2}
 \leq |\widehat \chi(\tau)|  \leq \frac{\kappa_2}{1+\tau^2}, \tau \in \RR,
\eeq
where $\kappa_2> \kappa_1>0$ are two fixed constants  that do not depend on $\tau$. 
We will show in the Appendix 
the existence of a such function. \\

\begin{theorem} \label{Tfrequency} Let $z_{0}\in X_2\setminus\{0\}$, and 
let  $z(t) = e^{itA}z_0$, and let $\widehat{x}(\tau)$ be the Fourier transform of 
$x(t)= \chi_T(t) z(t)$, where 
 $\chi_T(t)$ is the cut-off function defined by \eqref{cutoff},  and satisfying 
 the inequality \eqref{bbchi}. \\
 
 Then, there exists a constant $c_0= c_0(\chi)>0$ such that 
 the following inequality
\beq \label{lambdatau}
\lambda_1 \leq \lambda(\widehat x(\tau))\leq 4|\tau|+ c_0\lambda(z_0), 
\eeq
holds for all $ \tau \in \RR.$ 
\end{theorem}
\begin{proof}
Recall the expression of the frequency function:
\beq \label{lambda}
\lambda(\widehat x(\tau)) =  \langle A\widehat x(\tau) , \widehat x(\tau) 
\rangle_X \|\widehat x(\tau)\|_X^{-2},
\; \forall
\tau \in \RR.  
\eeq
Let $z_0= \sum_{k=1}^{+\infty} z_k \phi_k \in X_2$. Hence
\be
\widehat x(\tau) =  \sum_{k=1}^{+\infty} \widehat \chi_T(\tau-\lambda_k) z_k \phi_k.  
\ee
Hence
\be \label{ide}
\lambda(\widehat x(\tau)) =  \sum_{k=1}^{+\infty} \lambda_k |
\widehat \chi_T(\tau-\lambda_k)|^2 z_k^2 \left( \sum_{k=1}^{+\infty}
|\widehat \chi_T(\tau-\lambda_k)|^2 z_k^2 \right)^{-1}.  
\ee
We first remark that  $\lambda(\widehat x(\tau))\geq \lambda_1$ for all
$\tau \in \mathbb R$, and it 
 tends to  $\lambda(z_0)$
when $T$ approaches  $0$. In order to study the behavior of $\lambda(\widehat x(\tau))$
when $\tau$ is large we need to derive the behavior of $\widehat \chi_T(s)$ when
$s$ tends to infinity.  \\ 

We start with the trivial case where $\tau$ is far away from the spectrum of 
$A$, that is $\tau<\lambda_1$.

\medskip

Let $K \in \mathbb R_+$ be large enough, and set
\beqs
\sum_{k=1}^{+\infty}  |
\widehat \chi_T(\tau-\lambda_k)|^2 z_k^2 &=&  \sum_{|\tau-\lambda_k|\leq K}  |
\widehat \chi_T(\tau-\lambda_k)|^2 z_k^2 +\sum_{|\tau-\lambda_k|> K}  |
\widehat \chi_T(\tau-\lambda_k)|^2 z_k^2
= \mathcal I_1+  \mathcal I_2.
\eeqs

We claim that there exists $K_\tau>0$ large enough such that 
\be \label{inN}
2\mathcal I_2 \leq \mathcal I_1, \;\; \; \textrm{  for all  } K \geq K_\tau.
\ee
We first observe that there exists $r_0>0$ large enough such that 
\be \label{ll1}
 2\kappa_2\sum_{\lambda_k> r_0}  z_k^2 \leq 
  \kappa_1 \sum_{\lambda_k \leq r_0}  z_k^2,
\ee
or equivalently 

\beqs 
 \left(2\frac{\kappa_2}{\kappa_1}+1 \right)\sum_{\lambda_k> r_0}  z_k^2  \leq \|z_0\|_X^2.
 \eeqs

In fact, we have 
\beq \label{ww1}
\sum_{\lambda_k> r_0}  z_k^2 <  \frac{1}{r_0}
\sum_{\lambda_k> r_0} \lambda_k  z_k^2  \leq 
\frac{\lambda(z_0)}{r_0}\|z_0\|_X^2.
\eeq

Hence  the inequality \eqref{ll1}  holds if 
\be \label{r0}
r_0=2 \left(2\frac{\kappa_2}{\kappa_1}+1 \right)\lambda(z_0).
\ee  

 Now by taking $K=|\tau| +r_0$,  and using the bounds \eqref{bbchi} with
 $\widehat \chi_T(s)= T\widehat \chi(Ts)$ in mind, 
we get  
\beq
\label{ff1}2\mathcal I_2 \leq \frac{2\kappa_2T^2}{(1+K^2T^2)^2} \sum_{\lambda_k> K+\tau}  z_k^2,\\ 
\label{ff2} \mathcal I_1 \geq \frac{\kappa_1T^2}{(1+K^2T^2)^2} \sum_{\lambda_k\leq K+\tau}  z_k^2.
\eeq
Since $K\geq r_0$, inequalities \eqref{ll1}, \eqref{ff1} and  \eqref{ff2} imply
\beq
2\mathcal I_2 \leq \frac{\kappa_1T^2}{(1+K^2T^2)^2} \sum_{\lambda_k\leq K+\tau} 
\lambda_k z_k^2 \leq  \mathcal I_1.
\eeq
Then,  inequality \eqref{inN} is valid for $ K_\tau=\max(\tau, r_0)$. Consequently the inequalities

\beq \label{inNN}
\frac{1}{2} \mathcal I_1 \leq \|\widehat x(\tau)\|_X^{2}= \sum_{k=1}^{+\infty}  |
\widehat \chi_T(\tau-\lambda_k)|^2 z_k^2 \leq 3 \mathcal I_1,
\eeq
holds for all  $ K \geq K_\tau$.\\

Considering  now  identity \eqref{ide}, and  inequalities \eqref{inNN}, we obtain
\beqs \label{hhq}
 \lambda(\widehat x(\tau)) \leq  2  \left(\sum_{|\tau-\lambda_k|\leq K} \lambda_k |
\widehat \chi_T(\tau-\lambda_k)|^2 z_k^2\right) \left( \sum_{|\tau-\lambda_k|\leq K}
|\widehat \chi_T(\tau-\lambda_k)|^2 z_k^2 \right)^{-1}\\+  
2\left(\sum_{|\tau-\lambda_k|>K} \lambda_k |
\widehat \chi_T(\tau-\lambda_k)|^2 z_k^2\right) \left( \sum_{|\tau-\lambda_k|\leq K}
|\widehat \chi_T(\tau-\lambda_k)|^2 z_k^2 \right)^{-1} = \mathcal J_1 +\mathcal J_2.
\eeqs
On the other hand we have
\be \label{j1t}
\mathcal J_1 \leq 2(\tau +K).
\ee
In addition, using again the bounds \eqref{bbchi},  we obtain
\be \label{j2}
\mathcal J_2 \leq 2 \left(\sum_{\lambda_k> \tau +K} \lambda_k
z_k^2\right) \left( \sum_{\lambda_k\leq K+\tau }  z_k^2 \right)^{-1}. 
\ee

Since $K+\tau \geq r_0$, inequality  \eqref{ll1} gives 
\be \label{bn1}
\sum_{\lambda_k \leq K+\tau }  z_k^2 \geq  \left(\frac{\kappa_1}{2\kappa_2}
+1 \right)^{-1}\|z_0\|_X^2.
\ee

Hence

\be \label{jj2}
\mathcal J_2 \leq 2 \left(\frac{\kappa_1}{2\kappa_2}
+1 \right) \left(\sum_{k=1}^{+\infty} \lambda_k
z_k^2\right) \left( \sum_{k=1}^{+\infty}  z_k^2 \right)^{-1} = 2 \left(\frac{\kappa_1}{2\kappa_2}
+1 \right) \lambda (z_0).
\ee
Combining inequalities  \eqref{hhq}, \eqref{jj2}and \eqref{bn1}, we get 
\beqs 
 \lambda(\widehat x(\tau)) \leq  2|\tau| +2K +2\left(\frac{\kappa_1}{2\kappa_2}
+1 \right) \lambda(z_0).
 \eeqs
 for all  $ K \geq K_\tau$.\\

Consequently, the proof  is achieved by taking  
$c_0=   8 \frac{\kappa_2}{\kappa_1}+\frac{\kappa_1}{\kappa_2}+6.$

\end{proof}
 \begin{remark}
The upper bound of $\lambda(\widehat x(\tau))$ obtained 
in Theorem~\ref{Tfrequency} is not optimal since 
$\lambda(\widehat x(\tau)) = \lambda_k= \lambda(z_0)$ if $z_0 =\phi_k$.
Moreover when $\lambda_{max}(z_0) = \max \{\lambda_k, \; k \in \mathbb N^*,\; 
\langle z_0,\phi_k\rangle_X \not= 0\} < \infty$, we can easily show that  
$ \lambda(\widehat x(\tau)) \leq \lambda_{max}(z_0) $.  We remark that  in both
cases the bounds of $\lambda(\widehat x(\tau)) $ are independent of the Fourier
 frequency $\tau$.
 \end{remark}
 \begin{lemma} Let $c_0^\prime =  \frac{\|\dot \chi\|_{L^2(-1, 1)}}{\|\chi\|_{L^2(-1, 1)}},$
  $z_{0}\in X_2\setminus\{0\}$, and 
let  $z(t) = e^{itA}z_0$, and let $\widehat{x}(\tau)$ be the Fourier transform of $x(t)= \chi_T(t) z(t)$, where 
 $\chi_T(t)$ is the cut-off function defined by \eqref{cutoff}. \\
 
 Then,   the following inequality 
 \beq \label{inne1}
 \left(1- \frac{1}{R}\left( \frac{c_0^\prime}{T}  +\lambda(z_0)
  \right)\right) \|z_0\|_X^2 \leq \|\chi\|_{L^2(-1, 1)}^{-2} \int_{-R}^R\|\widehat{x}(\tau)\|_X^2 d\tau 
 \eeq
 holds for all $R >  \frac{c_0^\prime}{T}  +\lambda(z_0).$
 
 \end{lemma}
 \begin{proof}
Recall that 
$\dot x = f+iA x$ where $f= \dot \chi_T z$. By integration by parts we then have
\[
\widehat{x}(\tau) = -\frac{i}{\tau} \left(\widehat{f}(\tau)+ iA\widehat{x}(\tau) \right).
\]
Consequently 
\[
\|\widehat{x}(\tau)\|_X^2  =
\langle -\frac{i}{\tau} \left(\widehat{f}(\tau)+ iA\widehat{x}(\tau) \right), \widehat{x}(\tau)\rangle_X.
\]
Then for any $R>0$, by Fourier-Plancherel Theorem, we have
\beqs
  \|\chi\|_{L^2(-1, 1)}^2 \|z_0\|_X^2 
   \leq \int_{-R}^R\|\widehat{x}(\tau)\|_X^2 d\tau+
  \frac{1}{R}\left( \frac{1}{T} \|\dot \chi\|_{L^2(-1, 1)}\|\chi\|_{L^2(-1, 1)} +\lambda(z_0)\|\chi\|_{L^2(-1, 1)}^2 
   \right)\|z_0\|_X^2.
\eeqs
Hence for $R$ large enough we have 
\beqs
 \left(1- \frac{1}{R}\left( \frac{1}{T} 
 \frac{\|\dot \chi\|_{L^2(-1, 1)}}{\|\chi\|_{L^2(-1, 1)}} +\lambda(z_0)
  \right)\right) \|z_0\|_X^2 
   \leq \|\chi\|_{L^2(-1, 1)}^{-2}\int_{-R}^R\|\widehat{x}(\tau)\|_X^2 d\tau,
 \eeqs
which finishes the proof of the lemma.

\end{proof}

Back now to the proof of the theorem. Combining inequalities \eqref{inqq} and 
 \eqref{inne1}, we find
 
 \beq \label{important1}
    \left(1- \frac{1}{R}\left( \frac{c_0}{T}  +\lambda(z_0)
  \right)\right) \|z_0\|_X^2 \leq \|\chi\|_{L^2(-1, 1)}^{-2} \left( \int_{-R}^R
   \frac{\|{C \widehat x(\tau)}\|^{2}_Y}{\psi(\lambda(\widehat x(\tau)))} d\tau + 
 \int_{-R}^R\frac{\|\widehat{f}(\tau)\|^{2}_X}{\varepsilon(\lambda(\widehat x(\tau)))}d\tau\right).
\eeq
 
 Applying the upper bound $\lambda(\widehat x(\tau))$ derived in 
 Theorem~\ref{Tfrequency}, and considering the monotony of the functions
 $\psi$ and $\varepsilon$ in $\mathfrak C$, we obtain 
  \beqs 
    \left(1- \frac{1}{R}\left( \frac{c_0^\prime}{T}  +\lambda(z_0)
  \right)\right) \|z_0\|_X^2 \leq \frac{1}{\psi(4R+ c_0\lambda(z_0))}  
  \frac{\|\chi\|_{L^\infty(-1,1)}^2}{\|\chi\|_{L^2(-1, 1)}^{2}}
  \int_{0 }^T|{C z(t)}\|^{2}_Yd\tau \\
  +
 \frac{1}{T\varepsilon(4R+ c_0\lambda(z_0))} 
 \frac{\|\dot \chi\|_{L^2(-1, 1)}^2}{\|\chi\|_{L^2(-1, 1)}^{2}}\|z_0\|_X^2,
\eeqs
 for all $R >  \frac{c_0^\prime}{T}  +\lambda(z_0)$.\\
 
 Now, by taking $R= 2\left(\frac{c_0}{T}  +\lambda(z_0)\right)$, and 
 $\theta_0= \max(c_0^\prime, 8+c_0)$, we find
 
 \beqs 
  \left(1-    \frac{2}{T\varepsilon\left(\theta_0\left(\frac{1}{T}+\lambda(z_0)\right)\right)} 
 \frac{\|\dot \chi\|_{L^2(-1, 1)}^2}{\|\chi\|_{L^2(-1, 1)}^{2}}\right)\|z_0\|_X^2 
 \leq \frac{2}{\psi\left(\theta_0\left(\frac{1}{T}+\lambda(z_0)\right)\right)}  
  \frac{\|\chi\|_{L^\infty(-1,1)}^2}{\|\chi\|_{L^2(-1, 1)}^{2}}
  \int_{0 }^T\|{C z(t)}\|^{2}_Ydt.
\eeqs

Let $\theta_1=  \frac{4\|  \chi\|_{L^2(-1, 1)}^{2}}{\| \dot \chi\|_{L^\infty(-1,1)}^2}$, and 
$\theta_2= \frac{4\|\chi\|_{L^2(-1,1)}^2}{\| \chi\|_{L^\infty(-1, 1)}^{2}}
$.  \\

Then, for $T\varepsilon(4R+ c_0\lambda(z_0))\geq  \theta_1$,  we finally get the wanted estimate:

 \beq  \label{fee}
  \theta_2\psi\left(\theta_0\left(\frac{1}{T}+\lambda(z_0)\right)\right)   \|z_0\|_X^2 \leq 
   \int_{0 }^T\|{C z(t)}\|^{2}_Ydt.
\eeq
Simple calculation shows that the function $T\mapsto 
T\varepsilon \left(\theta_0\left(\frac{1}{T}+\lambda(z_0)\right)\right)$ is increasing,
tends to infinity when $T$ approaches $+\infty$, and tends  to $0$ when $T$ 
approaches $0$. Then there exists a unique value $T(\lambda(z_0))>0$ that solves
the equation \eqref{timeT}. In addition, the function  $\lambda \mapsto T(\lambda)$ is
 increasing. Finally, the inequality \eqref{fee} is valid for all $T \geq T(\lambda(z_0))$. \\

Now, we shall prove the converse. Our strategy is to adapt  the proof of Theorem 1.2 in
~\cite{RW} for the classical exact controllability to our settings (see also ~\cite{BZbb, Miller05}).
We further assume that the weak observability inequality \eqref{wobst}  holds for 
some fixed  $\psi$ and $\varepsilon$ in $\mathfrak C$. Our goal now 
is to show that  $C$ is indeed  spectrally coercive.\\

 Let $z_0 \in X_4$, and  $x_0:= (iA-i\tau I)z_0 $ for some $\tau \in \RR$. Define $x(t)= e^{itA}x_0$ and
 $z(t) = e^{itA}z_0$. \\
 
  A forward computation shows  that $z(t)$ solves the following 
 \beqs
 \dot z(t) -i\tau z(t) &=& x(t), \;\;\forall t\in \RR_+^*,\\
 z(0) &=& z_0.
 \eeqs
 
 Then 
 
 \[
 z(t) = e^{i\tau t} z_0 +\int_{0}^t e^{i\tau(t-s)} x(s) ds.
 \]
 Applying now the observability  operator  both sides gives
 \beqs
 Cz(t) = e^{i\tau t} Cz_0+ \int_{0}^t e^{i\tau(t-s)} Cx(s) ds,
\eeqs
 whence 
  \beqs
 \|Cz(t)\|_Y^2 \leq 2\|Cz_0\|_Y^2+2 \int_{0}^t \|Cx(s)\|_Y^2ds.
\eeqs

Integrating the  inequality above both sides over $(0, T)$, we obtain
\beqs
\int_0^T \|Cz(t)\|_Y^2 dt \leq 2T\|Cz_0\|_Y^2+2 T \int_0^T\| Cx(s) \|_Y^2 ds.
\eeqs

We deduce from the admissibility assumption \eqref{admm}
 that 
 \beqs
 \int_0^T \|Cz(t)\|_Y^2 dt \leq 2T\|Cz_0\|_Y^2+2 T C_T\|(A-\tau I)z_0\|^2_X.
 \eeqs

Applying the weak observability inequality \eqref{wobst} for $T= T(\lambda(z_0))$,
leads to 
\beqs
\theta_2\psi\left(\theta_0\left(\frac{1}{T(\lambda(z_0))}+\lambda(z_0)\right)\right)   \|z_0\|_X^2 
\leq 2T(\lambda(z_0))\|Cz_0\|_Y^2+2 T(\lambda(z_0)) C_{T(\lambda(z_0))}\|(A-\tau I)z_0\|^2_X,
\eeqs 
for all $\tau \in \mathbb R$. \\ 

Since $T(\lambda) \geq T_0 = T(0),$ for all $\lambda \geq 0$, we have
\beqs
\theta_2\psi\left(\theta_0\left(\frac{1}{T_0}+\lambda(z_0)\right)\right)   \|z_0\|_X^2 
\leq 2T(\lambda(z_0))\|Cz_0\|_Y^2+2 T(\lambda(z_0)) C_{T(\lambda(z_0))}\|(A-\tau I)z_0\|^2_X,
\eeqs

Taking $\tau = \lambda(z_0)$ in the previous inequality implies 
\beqs
\frac{\theta_2}{2 T(\lambda(z_0)) C_{T(\lambda(z_0))}}
\psi\left(\theta_0\left(\frac{1}{T_0}+\lambda(z_0)\right)\right)   \|z_0\|_X^2 
\leq \frac{1}{C_{T(\lambda(z_0))}} \|Cz_0\|_X^2+\|(A-   \lambda(z_0)I)z_0\|^2_X.
\eeqs 

Let
\beqs
\widetilde \psi(\lambda) = \frac{\theta_2}{4 T(\lambda) }\psi\left(\theta_0\left(\frac{1}{T_0}+\lambda\right)\right),\\
\widetilde \varepsilon(\lambda) = \frac{\theta_2}{4 T(\lambda) C_{\lambda}} \psi \left(\theta_0\left(\frac{1}{T_0}+\lambda \right)\right).
\eeqs

We deduce from the monotonicity properties  of $\psi(\lambda), C_\lambda, $ and $T(\lambda)$ that 
 $\widetilde \psi(\lambda),\, \widetilde \varepsilon(\lambda)
\in \mathfrak C $. \\

Consequently $C$ becomes  spectrally coercive with the functions
$\widetilde \psi(\lambda),\, \widetilde \varepsilon(\lambda)$, that is 

 \begin{eqnarray*}
 0\leq \frac{\|Az\|_X^2}{\|z\|_X^2} -\lambda^2(z) <\widetilde \varepsilon(\lambda(z)),
\end{eqnarray*}
implies
\begin{eqnarray*}
\|Cz\|_Y^2 \geq \widetilde \psi(\lambda(z)) \|z\|_{X }^2,
\end{eqnarray*}
which finishes the proof of the Theorem.

\section{Sufficient conditions for the spectral coercivity.} \label{sec4}

In this section we study the relation between  the  spectral coercivity  of the observability
operator $C$  given  in Definition~\ref{dwc},   and  the action of the operator $C$ on  vector spaces 
spanned by eigenfunctions associated to close eigenvalues.  \\

For  $\lambda \in \mathbb R_+$ and  $\varepsilon>0$, set
\be
N_\varepsilon( \lambda)= \{k \in \mathbb N^*  \textrm{ such that } |\lambda-\lambda_k |< \varepsilon\},
\ee
to be the index function of eigenvalues of  $A$ in a $\varepsilon$-neighborhood of  a given $\lambda$.\\

\begin{definition} \label{dsc}
The operator $C$ is weakly spectrally coercive if 
there exist a  constant $\varepsilon>0$  and a function 
 $\psi \in \cc$   
 such that 
for all $ \lambda \in \mathbb R$,
 the following inequality 
\be \label{swcc}
\|Cz\|_Y^2 \geq \psi(\lambda) \|z\|_{X }^2,
 \ee
 holds for all $z = \sum_{k \in N_\varepsilon(\lambda)} z_k \phi_k \in X_2 \setminus \left\{0\right\}.$

\end{definition}

\begin{lemma} \label{dsc2}
The operator $C$ is weakly spectrally coercive iff
there exist a  constant $\varepsilon>0$  and a function 
 $\psi \in \cc$   
 such that  the following inequality 
\be \label{swcc2}
\|Cz\|_Y^2 \geq \psi(\lambda_n) \|z\|_{X }^2,
 \ee
 holds for all $z = \sum_{k \in N_\varepsilon(\lambda_n)} z_k \phi_k,$ and for all
 $n\in \mathbb N^*$.
\end{lemma}
\begin{proof} 
Assume that $C$ is weakly spectrally coercive. 
By taking $\lambda = \lambda_n$ in  \eqref{swcc}, inequality  \eqref{swcc2}  immediately holds.  Conversely,
assume that inequality \eqref{swcc2}  is satisfied, and let $\lambda \in \mathbb R$.  One can easily check that
 the set $N_{\frac{\varepsilon}{2}}(\lambda)$ is either empty or  it contains at least an element $n_0 \in \mathbb N^*$. 
 Since $ N_{\frac{\varepsilon}{2}}(\lambda) \subset N_{\varepsilon}(\lambda_{n_0}),$ we have 
 \ba
\|Cz\|_Y^2 \geq \psi(\lambda_{n_0}) \|z\|_{X }^2,
 \ea
 holds for all $z = \sum_{k \in N_{\frac{\varepsilon}{2}}(\lambda)} z_k \phi_k \in X_2 \setminus \left\{0\right\}.$
 On the other hand the fact that  $\psi$ is non-increasing  implies 
  \ba
\|Cz\|_Y^2 \geq \psi \left(\lambda+\frac{\varepsilon}{2}\right) \|z\|_{X }^2,
 \ea
 holds for all $z = \sum_{k \in N_{\frac{\varepsilon}{2}}(\lambda)} z_k \phi_k \in X_2 \setminus \left\{0\right\},$ which
 shows that  $C$ is weakly spectrally coercive with the constant $\frac{\varepsilon}{2}>0$ and
$ \tilde \psi(\lambda):=  \psi(\lambda+\frac{\varepsilon}{2}) \in \cc$.

\end{proof}
The Lemma \ref{dsc2} has been proved in  \cite{RTTT} for the particular case where $\psi$ is a constant function.
\begin{theorem} \label{swc} 
Let $\varepsilon>0$ be a fixed constant and let $\psi \in \cc$. If $C$ is  spectrally coercive  with
$ \varepsilon, \psi$, then  it is  weakly spectrally coercive. Conversely, if $C$ is  weakly spectrally coercive  
with $ \varepsilon, \psi$, then   $C$ is  spectrally coercive.  
\end{theorem}
\begin{proof}

Let $\lambda \in \mathbb R_+, $  and $\beta>0$ being fixed.
A direct calculation
shows that if 
 \[  z = \sum_{k \in N_\beta(\lambda)} z_k \phi_k, \]
 we have
 
 \[
 (\lambda(z)-\lambda)\|z\|_X^2 = \sum_{k \in N_\beta(\lambda)} (\lambda_k -\lambda)z_k^2 .
 \]
 Hence 
 
 \[
 |\lambda(z)-\lambda| < \beta.
 \]
 On the other hand 
 \[
 \|Az\|_X^2  -\lambda^2(z)\|z\|_X^2 =\|(A-\lambda(z)I) z\|_X^2 \leq 2\|(A-\lambda I) z\|_X^2
  +2|\lambda-\lambda(z)|^2\|z\|_X^2 <2\beta^2\|z\|_X^2 .
 \]
Then,  we deduce from the  spectral coercivity in Definition~\ref{dwc}  that  \eqref{swcc} holds if we choose 
$\beta$ such that  $2\beta^2 < \eps$.

\medskip

Now, we shall prove the opposite implication.
 Assume that \eqref{swcc} is satisfied for all
$\lambda \in \mathbb R_+, $ and  let 
 \[  z = \sum_{k=1}^{+\infty} z_k \phi_k, \] 
being in $X_2\setminus\{0\}$, and satisfying  the inequality

\begin{eqnarray}\label{opp}
 \frac{\|Az\|_X^2}{\|z\|_X^2} -\lambda^2(z) <\beta(\lambda(z)),
 \end{eqnarray}
where $\beta$ will be chosen later in terms of $\eps$ and $\psi$. \\

Set 
\beq \label{opp2}
((A-\lambda(z)I) z= f.
\eeq
We deduce from~\eqref{opp}, the following estimate
\beq \label{opp3}
\|f\|_X^2 \leq  \beta\|z\|_X^2.
\eeq
We now introduce the following orthogonal decomposition of $z$: 
\beq \label{orthog}
z &=& z^0+\widetilde z, 
\eeq
with

\beq
z^0 =  \sum_{k \in N_\eps(\lambda(z))} z_k \phi_k,\;\; \widetilde z = 
 \sum_{k \notin N_\eps(\lambda(z))} z_k \phi_k.
\eeq
We deduce from \eqref{opp}, \eqref{opp2} and \eqref{opp3} the following estimate

\beq \label{luft1}
 \|\widetilde z \|_X^2 = \sum_{k \notin N_\eps(\lambda(z))} z_k^2 
 = \sum_{k \notin N_\eps(\lambda(z))} \frac{f_k^2}{(\lambda(z)-\lambda_k)^2} 
 \leq \frac{1}{\eps^2}\|f\|_X^2 \leq \frac{\beta}{\eps^2} \|z \|_X^2. 
\eeq
On the other hand the inequality \eqref{swcc} for $\lambda= \lambda(z)$ implies

\beq   \label{luft2}
 \|z^0\|_X^2\leq \frac{\|Cz^0\|_{Y}^2}{ \psi(\lambda(z))}.
\eeq

The following result has been proved
for admissible operator $C$  first on $(0, +\infty)$ in \cite{RW}, and 
on $(0, T)$ in \cite{RTTT}.
\begin{proposition} \label{RW}
For each  $\varepsilon>0$  
and  $\lambda \in \RR_+,$ we define the subspace $V(\lambda) \subset X$ by
\ba
V(\lambda) : = \left\{ \phi_k: \; k\notin N_\varepsilon(\lambda) \right\},
\ea
and we denote $A_\lambda:  V(\lambda) \cap X_2 \rightarrow X$, the restriction of the unbounded 
operator to $V(\lambda)$. \\

Then, there  exists a constant $M>0$, such that 
 \beq
 \|C(A_\lambda-\lambda I)^{-1}\|_{\mathcal L(V(\lambda), Y)} \leq M, \quad \forall \lambda \in \RR_+.
 \eeq

\end{proposition}
We deduce from \eqref{opp2} and \eqref{orthog},  the following inequality

\beq
\|Cz^0\|_{Y}^2 \leq   2\|Cz\|_{Y}^2+ 2\|C\widetilde z\|_{Y}^2
 \leq 2\|Cz\|_{Y}^2+2\| C(A_{\lambda(z)}-\lambda(z) I)^{-1}f\|_{Y}^2. \label{intineq}
\eeq
Applying now the results of Proposition \ref{RW} on \eqref{intineq}, we get

\be \label{innnn}
\|Cz^0\|_{Y}^2 \leq 2\|Cz\|_{Y}^2+2M\|f\|_{X}^2.
\ee
Inequalities \eqref{opp3} and \eqref{innnn}, give

\ba
\|Cz^0\|_{Y}^2 \leq 2\|Cz\|_{Y}^2+ 2M\beta \|z \|_X^2.
\ea

Now, using the inequality \eqref{luft2}, we get

\beq   \label{luft3}
 \|z^0\|_X^2\leq 2 \frac{\|Cz\|_{Y}^2}{ \psi(\lambda(z))}+\frac{2M\beta}{\psi(\lambda(z))} \|z \|_X^2. 
\eeq

Combining \eqref{luft1} and \eqref{luft3}, we obtain
\beqs
\|z\|_X^2 =  \|z^0\|_X^2 + \|\widetilde z \|_X^2 \leq 
\rho(\lambda(z))  \|z \|_X^2+ 2 \frac{\|Cz\|_{Y}^2}{\psi(\lambda(z))},
\eeqs

with 

\be
\rho(\lambda(z)): = \left(\frac{2M}{\psi(\lambda(z))}  +\frac{1}{\varepsilon}
\right)\beta(\lambda(z)). 
\ee
 By taking 
 \be
 \beta(\lambda(z)) := \frac{1}{2}\left(\frac{2M}{\psi(\lambda(z))}  +\frac{1}{\varepsilon}
\right)^{-1},
 \ee
 we find
 \beqs
  \frac{1}{4} \, \psi(\lambda(z)) \|z\|_X^2 \leq \|Cz\|_{Y}^2.
 \eeqs
 One can check easily that $\beta(\lambda)$ belongs to $\cc$. Then 
 $C$ becomes spectrally coercive with the  functions $ \beta(\lambda),  \frac{1}{4} \, \psi(\lambda) \in \cc$.

\end{proof}


\begin{remark}
Theorem \ref{swc} shows that the results of the paper \cite{RTTT} by M. Tucsnak {\it and al}. correspond to the 
particular  case of spectral coercivity where $\varepsilon$ and $\psi$ are constant functions. Finally, applying 
Proposition \ref{RW} is not necessary to prove the theorem. In fact we can bound in inequality \eqref{intineq},
$C$ by $\|C\|^2 (\lambda(z)+\varepsilon)^2\frac{\beta}{\eps^2} \|z \|_X^2$ where $ \|C\|$ is the norm of $C$  in 
$\mathcal L(X_2, Y).$ Applying the  results of  Proposition \ref{RW}  improves the behavior of $\varepsilon(\lambda)$
for large  $\lambda$.

\end{remark}

\section{Application to observability of the Schr\"odinger equation} \label{sec5}
Let   $\Omega = (0,  \pi) \times  (0,  \pi), $  and  $\partial \Omega$ be its boundary. 
We consider the following initial and boundary value problem:
  
\be
\label{syst0}
\left\{
\begin{array}{ll}
z^{ \prime }(x,t) +i \Delta z(x,t) = 0, \, x \in \Omega, \, t > 0,  \\

z(x, t) = 0, \, x\in \partial \Omega, \, t >0, \\
z(x, 0) = z_0(x), x\in \Omega.
\end{array}
\right.
\ee
Let $\Gamma$ be an open nonempty subset of $\partial \Omega$. Define 
$C $ to be the  following boundary  observability operator

\be\label{syst01}
y(x, t)=  C z(x,t) =  \partial_\nu z\vert_ \Gamma,
\ee
where $\nu$ is the outward  normal vector  on $\partial \Omega,$ and $ \partial_\nu $ is the normal 
derivative.

\medskip
 
We further show that 
the  observation  system \eqref{syst0}-\eqref{syst01} fits perfectly in the general formulation of the
system~\eqref{maineq}.  \\
  
  Let $X = H^1_0(\Omega)$  be the Hilbert space with  scalar
  product 
  \ba
  \langle v, w\rangle_X = \int_{\Omega} \nabla u \cdot \nabla v \, dx. 
  \ea
  
 Therefore $A= -\Delta: X_2 \subset X \rightarrow X$, is a  linear unbounded self-adjoint, strictly 
positive operator with a compact resolvent. Hence the 
operator $iA$ generates a strongly continuous group of isometries in $X$
denoted $(e^{itA})_{t\in \mathbb{R}}$.  Moreover  for $\beta \geq 0$, $X_\beta = D(A^{\frac{\beta}{2}})$
is given by
 
 \ba
 X_\beta= \left\{\phi \in  H^1_0(\Omega): \, (-\Delta)^{\frac{\beta}{2}} \phi \in  H^1_0(\Omega) \right\}.
 \ea
Then the  observability operator $C:  X_2 \rightarrow Y:=L^2(\Gamma),$ defined by \eqref{syst01}, is 
a bounded operator. In addition it is known  that $C$  is an admissible observability operator, that is  for 
any $T>0$ there exists a constant $C_T>0$, such that the following inequality holds 
\ba
\int_0^T \int_{\Gamma}\left|  \partial_\nu  z\right|^2 ds(x) dt \leq C_T^2 
\int_{\Omega}  \left|\nabla z_0\right|^2 dx,
 \ea
for all $z_0 \in X_2$.\\

The eigenvalues of $A$ are 
\be \label{eigenva}
\lambda_{m, n} = m^2+n^2, \;\; m, n  \in \mathbb N^*.
\ee
A corresponding family of normalized  eigenfunctions in  $ H^1_0(\Omega)$ are

\be \label{eigenfu}
\phi_{m,n}(x) = \frac{2}{\pi \sqrt{m^2+n^2}} \sin(n\pi x_1)\sin(m\pi x_2),  \;\; m, n  \in \mathbb N^*,\;\;
x=(x_1, x_2) \in \Omega.
\ee

Next we derive  observability inequalities corresponding to  different geometrical  assumptions
on the observability set $\Gamma$.

\medskip

{\it  Assumption {\bf I}}:  We  assume that  $\Gamma$  contains at least two touching sides of $\Omega$. \\

In this case it is known that $\Gamma$ satisfies the geometrical assumptions of 
\cite{BLR92}, and the exact controllability is reached \cite{Le}. We will show that it is indeed
the situation by applying our coercivity test.\\

Consider the Helmholtz equation defined by

\be
\label{usystem}
\left\{
\begin{array}{ll}
\Delta u +k^2u = f, \, x \in \Omega,  \\
u = 0, \, x\in \partial \Omega\setminus \overline \Gamma, \\
\partial_\nu u -ik u =0, \, x\in \Gamma, 
\end{array}
\right.
\ee
where $g \in L^2(\Gamma)$ and $f\in L^2(\Omega)$.

\medskip

 It has been shown using  Rellich's 
identities (which are somehow related to the multiplier approach in observability \cite{Li,Li1})  
the following result \cite{He}.
\begin{proposition} \label{Pfrequency} Under the assumptions {\bf I} on $\Gamma$, a  solution $u \in H^1(\Omega)$ to the system
\eqref{usystem} satisfies the following inequality
\be
k\|u\|_{L^2(\Omega)} + \|\nabla u\|_{L^2(\Omega)} \leq c_0 \left(\|f\|_{L^2(\Omega)} +\|g\|_{L^2(\Gamma)}\right),
\ee
for all $k \geq k_0$, where $k_0>0$ and  $c_0>0$ are constants that only depend on $\Gamma$.
\end{proposition}

We deduce from Proposition \ref{Pfrequency} the following inequality 
\be
\|z\|_X  \leq   c_1\left( \|Az -\lambda(z) z\|_{X} + \|C z\|_{Y}\right),  
\ee
for all $z \in X_2\setminus \{0\},$ where $\lambda(z)$ is the $A$-frequency of $z$, and
 $c_1>0$ are constants that only depend on $\Gamma$.
Then by taking $\varepsilon (\lambda) = \frac{1}{4c_1^2}$, we find
that $C$ is spectrally coercive with $\psi(\lambda) =\frac{1}{4c_1^2}$, which implies in turn that the system 
 \eqref{syst0}-\eqref{syst01} is exactly observable.\\
 
\begin{theorem} \label{mainsquare1}
Under the assumptions {\bf I} on $\Gamma$, the  system \eqref{syst0}-\eqref{syst01}, is exactly observable.
\end{theorem}

{\it  Assumption {\bf II}}:  We  assume that  $\Gamma$  in a one side of $\Omega$. Without
 loss of generality, we further assume  that  $ \Gamma= (0, \pi) \times \{0\}$. 

\medskip

The following result has been derived  partially in \cite{BaK2}.
\begin{proposition} \label{frequency2} Under the assumptions {\bf II} on $\Gamma$, a  
solution $u \in H^1(\Omega)$ to the system
\eqref{usystem} satisfies the following inequality
\be
k\|u\|_{L^2(\Omega)} + \|\nabla u\|_{L^2(\Omega)} \leq c_0 k\left(\|f\|_{L^2(\Omega)} + 
 \|g\|_{L^2(\Gamma)}\right),
\ee
for all $k \geq k_0$, where $k_0>0$ and  $c_0>0$ are constants that only depend on $\Gamma$.
\end{proposition}

We again deduce from Proposition \ref{frequency2} the following resolvent inequality 
\be
\|z\|_X  \leq   c_1(1+\sqrt{\lambda(z)}) \left( \|Az -\lambda(z) z\|_{X} + \|C z\|_{Y}\right),  
\ee

for all $z \in X_2\setminus \{0\},$ where $\lambda(z)$ is the $A$-frequency of $z$, and 
$c_1>0$ is a constant that only depends on $\Gamma$.
Then by taking $\varepsilon (\lambda) = \frac{1}{8c_1^2(1+\lambda)}$, we find that $C$ is spectrally coercive 
with $\psi(\lambda) =\frac{1}{8c_1^2(1+\lambda)}$.
This  implies in turn that the system 
 \eqref{syst0}-\eqref{syst01} is weakly observable: there exists a constant $T^0>0$  such that 
 \be \label{obsRec}
 \psi(\lambda(z_0)) \|z_0\|^2_{H^1_0(\Omega)} \leq \int_0^T  \int_{\Gamma}\left|  \partial_\nu  z\right|^2 ds(x) dt,
 \ee
 for all $z_0 \in X_2$, and for all $T\geq T^0$.\\

\begin{theorem} \label{mainsquare2}
Under the assumptions {\bf II} on $\Gamma$, the  system \eqref{syst0}-\eqref{syst01}, is weakly observable for any $z_0 \in X$.
\end{theorem}

 \medskip
 
 {\it  Assumption {\bf III}}:  We  assume that  $\overline{\Gamma}$ is included 
 in a one side of $\Omega$. Without
 loss of generality, we further assume  that  $ (\alpha, \beta)\times\{0\} \subset  \Gamma
 \subset  (0, \pi) \times \{0\}$, with $0<\alpha<\beta<\pi.$  Then, we have the following weak
 observability inequality.
 
\begin{theorem} \label{mainsquare3}
Under the assumptions {\bf III} on $\Gamma$, the  system \eqref{syst0}-\eqref{syst01}, is weakly observable for any $z_0 \in X$ with $\widetilde \varepsilon(\lambda) = \frac{1}{ \frac{4M}{\delta_\Gamma} \lambda  +1}$
 and $\widetilde \psi(\lambda) = \frac{\delta_\Gamma}{4 \lambda},$ where $\delta_\Gamma>0$ is a constant that only depends on $\Gamma$, and $M>0$ is the  admissibility constant appearing in Proposition \ref{RW}.  

\end{theorem}

 Different from the proofs in the two first cases, the proof of the weak observability in the theorem above is
 based on intrinsic properties of the eigenelements of $A$ and the operator  $C$. We first present the following 
useful  result.
 \begin{lemma} \label{wscsquare}
 The operator $C$ is weakly spectrally coercive, that is,
 the following inequality 
\be \label{ttt1}
\|Cz\|_Y^2 \geq \psi(\lambda_{m,n}) \|z\|_{X }^2,
 \ee
 holds for all $z = \sum_{k \in N_{\frac{1}{2}}(\lambda_{m,n})} z_k \phi_k $
 where $\psi(\lambda) =\frac{\delta_\Gamma}{\lambda},$ with $\delta_\Gamma>0$ is a constant that only depends on $\Gamma$. 
 
\end{lemma}
\begin{proof}
Let $\lambda_{m, n} = m^2+n^2$ be fixed eigenvalue, and let 
$z = \sum_{k \in N_{\frac{1}{2}}(\lambda_{m,n})} z_k \phi_k $ be a fixed vector in $X_2\setminus\{0\}$.\\

 It is easy to check that 
\be
N_{\frac{1}{2}}(\lambda_{m,n})= \{k = (p, q) \in \mathbb N^*\times \mathbb N^*:  p^2+q^2 = m^2+n^2\}.
\ee
Therefore

\beq 
\|Cz\|_Y^2&=&  
\int_\Gamma\left |  \sum_{k \in N_{\frac{1}{2}}(\lambda_{m,n})} z_k C\phi_k(x)\right |^2 ds(x), \nonumber\\
&\geq& \frac{4}{\pi^2} \int_\alpha^\beta \left |  \sum_{p^2+q^2=  m^2+n^2}  
\frac{q}{(p^2+q^2)^{\frac{1}{2}}} z_{p, q}\sin(px_1) \right |^2 dx_1,\label{rrrH}
\eeq

Based on techniques related to nonharmonic Fourier series, the following inequality has been
proved in Proposition 7 of \cite{RTTT}. 

\beq \label{rrrH2}
 \int_\alpha^\beta \left |  \sum_{p^2+q^2=  m^2+n^2}  
\frac{q}{(p^2+q^2)^{\frac{1}{2}}}z_{p, q}\sin(px_1) \right |^2 dx_1 \geq \tilde \delta_{\alpha, \beta}
\sum_{p^2+q^2=  m^2+n^2}\frac{q^2}{p^2+q^2}|z_{p, q}|^2, 
\eeq
where $ \tilde \delta_{\alpha, \beta}>0$ only depends on $\alpha$ and $\beta$. \\

Combining now inequalities \eqref{rrrH} and
\eqref{rrrH2}, we find 

\beqs
\|Cz\|_Y^2 \geq  \delta_\Gamma
\sum_{p^2+q^2=  m^2+n^2}\frac{q^2}{p^2+q^2}|z_{p, q}|^2
 \geq  \frac{\delta_\Gamma}{\lambda_{m,n}} \|z\|_X^2,
\eeqs
which achieves the proof. Here $ \delta_\Gamma :=  \frac{4}{\pi^2} \tilde \delta_{\alpha, \beta}$ only
depends on $\Gamma$.

\end{proof}


\begin{proof}[Proof of Theorem \ref{mainsquare3}.]

The result of the theorem is a direct consequence of  Lemma~\ref{dsc2}, Theorem~\ref{swc}, and
Lemma~\ref{wscsquare}. We finally obtain that  $C$ is spectrally coercive 
with  $\widetilde \varepsilon(\lambda) = \frac{1}{2}\left(\frac{2M}{ \psi(\lambda)}  +1
\right)^{-1} $
 and $\widetilde \psi(\lambda) = \frac{1}{4}\psi(\lambda)$, which finishes the proof.

\end{proof}
\begin{remark}
We observe that the result of Theorem \ref{mainsquare2} based on clever
analysis of Fourier series derived in \cite{BaK2},  is indeed  a particular case of Theorem \ref{mainsquare3}
($\alpha=0$ and $\beta =\pi$) obtained from Ingham type inequalities. 

\end{remark}
\section*{Appendix} 
\label{appendixA}

Let $\chi \in C_0(\RR)$ be a cut off function with a
compact support in $(-1, 1)$ given by
\beqs
\chi(s) = (1-|s|)e^{-2|s|}\mathbbm{1} _{(-1,1)}. 
\eeqs
Then we have the following result.

\begin{proposition} The function $\chi(s)$ satisfies 
\beqs
\chi \in H^1_0(-1,1),\;  \frac{\kappa_1}{1+\tau^2}
 \leq |\widehat \chi(\tau)|  \leq \frac{\kappa_2}{1+\tau^2}, \tau \in \RR,
\eeqs
where $\kappa_1> \kappa_2$ are two fixed constants. 
\end{proposition}

\begin{proof} Since $|\widehat \chi(\tau)| $ is even we shall prove the 
inequality only for $\tau \in \RR_+$.\\

A forward computation  gives
\[
\widehat \chi(\tau) = \frac{2}{1+\tau^2} +2\Re\left(\frac{1-e^{-(1+i\tau)}}{ (1+i\tau)^2}\right).
\]
Then
 
\[
|\widehat \chi(\tau)| \leq  \frac{6}{1+\tau^2}.
\]
On the other hand,  we have 
\beqs
\widehat \chi(\tau) = \int_\RR \frac{2}{1+(s-\tau)^2} 
  \mbox{sinc}^2 \left(\frac{s}{2}\right) ds.
  \eeqs
Using the estimate  $\mbox{sinc(s)} \geq \frac{2}{\pi}$ for $s\in (0, \frac{\pi}{2})$, we get
 \beqs
 \widehat \chi(\tau)  \geq \frac{4}{\pi}\int_{-\frac{1}{4}}^{\frac{1}{4}} \frac{1}{1+(s-\tau)^2} ds\\
   \geq \frac{4}{\pi} (\arctan(\frac{1}{4}-\tau)+\arctan(\frac{1}{4}+\tau)) = 
   \frac{4}{\pi} \arctan \left({\frac{\frac{1}{2}}{\frac{15}{16}+\tau^2}} \right)\\
   \geq \frac{4}{\pi}\left[ \frac{\frac{1}{2}}{\frac{15}{16}+\tau^2} -
    \frac{1}{3} \left(\frac{\frac{1}{2}}{\frac{15}{16}+\tau^2}\right)^3\right]
   \\ \geq   \frac{\frac{4}{3\pi}}{\frac{15}{16}+\tau^2}
   \\\geq \frac{\frac{4}{3\pi}}{1+\tau^2},
\eeqs
which finishes the proof.
\end{proof}

\section*{Acknowledgements}
FT was supported by the grant ANR-17-CE40-0029 of the French National Research Agency ANR (project MultiOnde).


\end{document}